\documentclass[12pt,twoside]{amsart}
\usepackage{amsmath, amsthm, amscd, amsfonts, amssymb, graphicx}
\usepackage{enumerate}
\usepackage[colorlinks=true,
linkcolor=blue,
urlcolor=cyan,
citecolor=red]{hyperref}
\usepackage{mathrsfs}
\newtheorem{theorem}{Theorem}[section]

\newtheorem{remark}{Remark}[section]

\newtheorem{corollary}{Corollary}[section]

\numberwithin{equation}{section}

\begin{document}
	
\title{More About operator order preserving}
\author{Gholamreza Karamali, Hamid Reza Moradi and Mohammad Sababheh}

\begin{abstract}
It is well known that increasing functions do not preserve operator order in general; nor do decreasing functions reverse operator order. However, operator monotone increasing or operator monotone decreasing do. In this article, we employ a convex approach to discuss operator order preserving or conversing. As an easy consequence of more general results, we find non-negative constants $\gamma$ and $\psi$ such that $A\leq B$  implies
$$f(B)\leq f(A)+\gamma {\bf{1}}_{\mathcal{H}}\;~{\text{and}}~\;f(A)\leq f(B)+\psi {\bf{1}}_{\mathcal{H}},$$ for the self adjoint operators $A,B$ on a Hilbert space $\mathcal{H}$ with identity operator ${\bf{1}}_{\mathcal{H}}$ and for the convex function $f$  whose domain contains the spectra of both $A$ and $B$.

The connection of these results to the existing literature will be discussed and the significance will be emphasized by some examples.
 
\end{abstract}
\maketitle
\pagestyle{myheadings}
\markboth{\centerline {More About Operator Order Preserving}}
{\centerline {G. Karamali, H. R. Moradi \& M. Sababheh}}
\bigskip
\bigskip
\section{Introduction}
Let $\mathcal{B}\left( \mathcal{H} \right)$ be the $C^*$--algebra of all bounded linear operators on a Hilbert space $\mathcal{H}$ and let ${{\mathbf{1}}_{\mathcal{H}}}$ be the identity operator in  $\mathcal{B}\left( \mathcal{H} \right)$.\\
An operator $A\in\mathcal{B}(\mathcal{H})$ is said to be positive (written $A\ge0$) if $\left\langle Ax,x \right\rangle \ge 0$ holds for all $x\in \mathcal{H}$.  If $A$ is positive and invertible, $A$ is said to be strictly positive and is written as $A>0$. This positivity defines a partial order on the class of self adjoint operators, where we write $B\ge A$ in case $B-A\ge0$.\\
 The Gelfand map $f\left( t \right)\mapsto f\left( A \right)$ is an isometrical $*$--isomorphism between the ${{C}^{*}}$--algebra $C\left( \sigma \left( A \right) \right)$ of continuous functions on the spectrum $\sigma \left( A \right)$ of a self adjoint operator $A$ and the ${{C}^{*}}$--algebra generated by $A$ and the identity operator ${{\mathbf{1}}_{\mathcal{H}}}$. This is called the functional calculus of $A$.  If $f,g\in C\left( \sigma \left( A \right) \right)$, then $f\left( t \right)\ge g\left( t \right)$ ($t\in \sigma \left( A \right)$) implies  $f\left( A \right)\ge g\left( A \right)$.

For $A,B\in \mathcal{B}\left( \mathcal{H} \right)$, $A\oplus B$ is the operator defined on $\mathcal{B}\left( \mathcal{H}\oplus \mathcal{H} \right)$ by $\left( \begin{matrix}
A & 0  \\
0 & B  \\
\end{matrix} \right)$. 
 A linear map $\Phi:\mathcal{B}\left( \mathcal{H} \right)\to \mathcal{B}\left( \mathcal{K} \right)$ is positive if $\Phi \left( A \right)\ge 0$ whenever $A\ge 0$. It's said to be unital if $\Phi \left( {{\mathbf{1}}_{\mathcal{H}}} \right)={{\mathbf{1}}_{\mathcal{K}}}$. 

Among the most interesting functions when studying the $C^*$--algebra $\mathcal{B}(\mathcal{H})$ are the so called operator convex and operator monotone.  A continuous function $f$ defined on the interval $J$
is called an operator convex function if $f\left( \left( 1-v \right)A+vB \right)\le \left( 1-v \right)f\left( A \right)+vf\left( B \right)$ for every $0<v<1$ and for every pair of bounded self-adjoint operators $A$ and $B$ whose spectra are both in $J$. On the other hand, $f$ is said to be operator monotone if $A\leq B$ implies $f(A)\leq f(B)$ for every pair of such self adjoint operators with spectra in $J.$

It is well known that a convex function is not necessarily operator convex and a monotone function is not necessarily operator monotone. The function $f:\mathbb{R}\to\mathbb{R}$ defined by $f(t)=t^3$ serves as an example for both purposes; \cite[Example V.1.4., p. 114]{bhatia}.

The celebrated Choi inequality \cite{choi,davis} states that an operator convex function $f:J\to\mathbb{R}$ and a unital positive linear mapping $\Phi$ satisfy
\begin{equation}\label{choi_ineq_intro}
f(\Phi(A))\leq \Phi(f(A)),
\end{equation}
for any self adjoint operator $A$ with spectrum in $J.$\\
Hansen et al. \cite{1} extended Choi's inequality \eqref{choi_ineq_intro} and  showed that if $f:J\to \mathbb{R}$ is an operator convex function, ${{A}_{1}},\ldots ,{{A}_{n}}\in \mathcal{B}\left( \mathcal{H} \right)$ are self-adjoint operators with  spectra in $J$, and  ${{\Phi }_{1}},\ldots ,{{\Phi }_{n}}:\mathcal{B}\left( \mathcal{H} \right)\to \mathcal{B}\left( \mathcal{K} \right)$ are positive linear mappings such that $\sum\nolimits_{i=1}^{n}{{{\Phi }_{i}}\left( {{\mathbf{1}}_{\mathcal{H}}} \right)}={{\mathbf{1}}_{\mathcal{K}}}$, then 
\begin{equation}\label{7}
f\left( \sum\limits_{i=1}^{n}{{{\Phi }_{i}}\left( {{A}_{i}} \right)} \right)\le \sum\limits_{i=1}^{n}{{{\Phi }_{i}}\left( f\left( {{A}_{i}} \right) \right)}.
\end{equation}
It is evident that a convex function (not operator convex) does not necessarily satisfy \eqref{choi_ineq_intro} nor \eqref{7}. However, if $f$ is convex, the following weaker inequality holds \cite[Lemma 2.1]{5}
\begin{equation}\label{19}
f\left( \left\langle \sum\limits_{i=1}^{n}{{{\Phi }_{i}}\left( {{A}_{i}} \right)}x,x \right\rangle  \right)\le \left\langle \sum\limits_{i=1}^{n}{{{\Phi }_{i}}\left( f\left( {{A}_{i}} \right) \right)}x,x \right\rangle
\end{equation}
for any unit vector $x\in \mathcal{K}$. Such inequalities are called  Jensen operator inequalities. We refer the reader to \cite{10, 6, re2, 11, re1, re3} for some references treating these inequalities.

The first goal of this article is to present convex versions of \eqref{choi_ineq_intro} and \eqref{7}. Of course, this will imply weaker inequalities than those for operator convex functions.

Strongly related to this, we discuss operator order preserving or conversing under monotone functions that are not necessarily operator monotone. 

Recall that the function $f:[0,\infty)\to [0,\infty)$ defined by $f(t)=t^p$ is operator monotone if and only if $0\leq p\leq 1;$ \cite[Corollary 1.16]{book1}. This implies the celebrated  \lq\lq L\"owner-Heinz inequality'' which asserts that if $0\le A\le B$, then ${{A}^{p}}\le {{B}^{p}}$ for any $p\in \left[ 0,1 \right]$. Now since if $p>1$, the function $f(t)=t^p$ is not operator monotone, the L\"owner-Heinz inequality does not  hold for $p>1$. Related to this problem, Furuta \cite{n3} proved the following ``order preserving result". In this result and in what follows, $\sigma(A)$ denotes the spectrum of $A$.
\begin{theorem}\label{2}
	Let $A,B\in \mathbb{B}\left( \mathcal{H} \right)$ be two positive operators such that $\sigma \left( A \right)\subseteq \left[ m,M \right]$  for some scalars $0<m<M$. If $B\le A$, then 
	\[{{B}^{p}}\le {{K}}\left( m,M,p \right){{A}^{p}}\quad\text{ for }p\ge 1,\] 
	where $K\left( m,M,p \right)$ is the generalized Kantorovich constant defined by
	\begin{equation}\label{9}
	K\left( m,M,p \right)=\frac{(m{{M}^{p}}-M{{m}^{p}})}{\left( p-1 \right)\left( M-m \right)}{{\left( \frac{p-1}{p}\frac{{{M}^{p}}-{{m}^{p}}}{m{{M}^{p}}-M{{m}^{p}}} \right)}^{p}}\quad\text{ for }p\in \mathbb{R}.
	\end{equation}
\end{theorem}
The role of Theorem \ref{2} is to present a L\"owner-Heinz inequality for $p>1$. However, this holds at the cost of an additional constant ${{K}}\left( m,M,p \right).$

It is clear that Theorem \ref{2} is an attempt to extend the relation
\begin{equation}\label{oper_mon_ineq_intro}
A\leq B\Rightarrow f(A)\leq f(B)
\end{equation}
 valid for the operator monotone function $f$ to the context of a monotone function; namely $f(t)=t^p, p>1.$ Elegantly, the relation \eqref{oper_mon_ineq_intro} was extended in \cite[Theorem 2.1]{n1} to a more general form that also extends Theorem \ref{2}; as follows.
\begin{theorem}\label{th1.2}
	Let $A,B\in \mathbb{B}\left( \mathcal{H} \right)$ be two positive operators satisfying $\sigma \left( A \right),\sigma \left( B \right)\subseteq \left[ m,M \right]$ for some scalars $0<m<M$ and let $f:\left[ m,M \right]\to \mathbb{R}$ be an increasing convex function. If $B\le A$, then for a given $\alpha >0$,
	\begin{equation}\label{5}
f\left( B \right)\le \alpha f\left( A \right)+\beta {{\mathbf{1}}_{\mathcal{H}}},	
	\end{equation}
	holds for
	\begin{equation}\label{12}
	\beta =\underset{m\le t\le M}{\mathop{\max }}\,\left\{ {{a}_{f}}t+{{b}_{f}}-\alpha f\left( t \right) \right\},
	\end{equation}
	where
	\[{{a}_{f}}\equiv \frac{f\left( M \right)-f\left( m \right)}{M-m}\quad\text{ and }\quad{{b}_{f}}\equiv \frac{Mf\left( m \right)-mf\left( M \right)}{M-m}.\]
\end{theorem}
Of course, the case $f\left( t \right)={{t}^{p}}\left( p\ge 1 \right)$ in Theorem \ref{th1.2} reduces to Theorem \ref{2} (see \cite[Remark 3.3]{n1}).\\ 
Further, the following converse of Theorem \ref{th1.2} has been proven in \cite[Theorem 2.1]{n2}.
\begin{theorem}\label{c}
	Let $A,B\in \mathbb{B}\left( \mathcal{H} \right)$ be two positive operators satisfying $\sigma \left( A \right),\sigma \left( B \right)\subseteq \left[ m,M \right]$ for some scalars $0<m<M$ and let $f:\left[ m,M \right]\to \mathbb{R}$ be a decreasing convex function. If $B\le A$, then for a given $\alpha >0$,
	\begin{equation}\label{11}
	f\left( A \right)\le \alpha f\left( B \right)+\beta {{\mathbf{1}}_{\mathcal{H}}},
	\end{equation}
	holds with $\beta$ as in \eqref{12}.
\end{theorem}
We here cite \cite{n5} and \cite{n4} as pertinent references to inequalities of types \eqref{5} and \eqref{11}.\\
Pondering both Theorems \ref{th1.2} and \ref{c} and their proofs we find the roles of both convexity and monotony of $f$. Further, we find the condition that $\sigma(A),\sigma(B)\subset [m,M]$ for some finite scalars $m,M$ are necessary.

In this article, we present convex and monotone versions of \eqref{choi_ineq_intro} and \eqref{7} without appealing to operator convexity or monotony, and naturally this will be at the cost of additional constants. Further and as important, we present new versions of \eqref{5} and \eqref{11} for self adjoint operators $A$ and $B$, without having the restriction $\sigma(A),\sigma(B)\subset [m,M].$

For example, in Theorem \ref{05}, we find two quantities $I_1$ and $I_2$ such that
$$I_1\leq \left\langle \sum\limits_{i=1}^{n}{{{\Phi }_{i}}\left( f\left( {{A}_{i}} \right) \right)}x,x \right\rangle -\left\langle f\left( \sum\limits_{i=1}^{n}{{{\Phi }_{i}}\left( {{B}_{i}} \right)} \right)x,x \right\rangle\leq I_2;$$ as an extension and reverse of \eqref{19}. Then several consequences  of the aforementioned inequalities are deduced.

The gradient inequality for convex functions stating that
\begin{equation}\label{grad_ineq_intro}
f(s)+f'(s)(t-s)\leq f(t),~ s,t\in J,
\end{equation}
for the differentiable convex $f:J\to\mathbb{R}$ will be a main tool in our proofs. We emphasize that this approach is a new approach in obtaining such inequalities.

\section{Main Results}
We present our main results in two parts; where in the first part we present the operator order results without prior assumptions on the order between $A$ and $B$. Then, in the second part we clarify how these results work when assuming an order like $A\leq B$.

However, all results presented in this direction will follow from a very general result, that we prove first as a generalized form of \eqref{19}. For the rest of this paper, the notation $C^{1}_{\text{co}}(J)$ will stand for the class of differentiable convex functions defined on the interval $J$.
\begin{theorem}\label{05}
	Let ${{\Phi }_{1}},\ldots ,{{\Phi }_{n}}:\mathcal{B}\left( \mathcal{H} \right)\to \mathcal{B}\left( \mathcal{K} \right)$ be positive linear mappings with $\sum\nolimits_{i=1}^{n}{{{\Phi }_{i}}\left( {{\mathbf{1}}_{\mathcal{H}}} \right)}={{\mathbf{1}}_{\mathcal{K}}}$ and let ${{A}_{1}},\ldots ,{{A}_{n}},~{{B}_{1}},\ldots ,{{B}_{n}}\in \mathcal{B}\left( \mathcal{H} \right)$ be self adjoint operators with  spectra contained in the interval $J$. If $f\in {{C}^{1}_{\text{co}}}\left( J \right)$, then for any unit vector $x\in \mathcal{K}$,
	\begin{align}
	& \left\langle \sum\limits_{i=1}^{n}{{{\Phi }_{i}}\left( {{A}_{i}} \right)}x,x \right\rangle \left\langle f'\left( \sum\limits_{i=1}^{n}{{{\Phi }_{i}}\left( {{B}_{i}} \right)} \right)x,x \right\rangle -\left\langle f'\left( \sum\limits_{i=1}^{n}{{{\Phi }_{i}}\left( {{B}_{i}} \right)} \right)\sum\limits_{i=1}^{n}{{{\Phi }_{i}}\left( {{B}_{i}} \right)}x,x \right\rangle  \nonumber\\ 
	& \le \left\langle \sum\limits_{i=1}^{n}{{{\Phi }_{i}}\left( f\left( {{A}_{i}} \right) \right)}x,x \right\rangle -\left\langle f\left( \sum\limits_{i=1}^{n}{{{\Phi }_{i}}\left( {{B}_{i}} \right)} \right)x,x \right\rangle  \label{02}\\ 
	& \le \left\langle \sum\limits_{i=1}^{n}{{{\Phi }_{i}}\left( f'\left( {{A}_{i}} \right){{A}_{i}} \right)}x,x \right\rangle -\left\langle \sum\limits_{i=1}^{n}{{{\Phi }_{i}}\left( f'\left( {{A}_{i}} \right) \right)}x,x \right\rangle \left\langle \sum\limits_{i=1}^{n}{{{\Phi }_{i}}\left( {{B}_{i}} \right)}x,x \right\rangle   \label{3}.
	\end{align}
\end{theorem}
\begin{proof}
	Since $f$ is convex and differentiable on $J$,  it follows from \eqref{grad_ineq_intro} that 
	\[f'\left( s \right)\left( t-s \right)\le f\left( t \right)-f\left( s \right)\le f'\left( t \right)\left( t-s \right)\] 
	for any $t,s\in J$. This is equivalent to
	\begin{equation}\label{1}
	f'\left( s \right)t-f'\left( s \right)s\le f\left( t \right)-f\left( s \right)\le f'\left( t \right)t-f'\left( t \right)s.
	\end{equation}
	Applying functional calculus for the operator ${{A}_{i}}\left( i=1,\ldots ,n \right)$, we have
	\[f'\left( s \right){{A}_{i}}-f'\left( s \right)s{{\mathbf{1}}_{\mathcal{H}}}\le f\left( {{A}_{i}} \right)-f\left( s \right){{\mathbf{1}}_{\mathcal{H}}}\le f'\left( {{A}_{i}} \right){{A}_{i}}-sf'\left( {{A}_{i}} \right).\] 
	Applying the positive linear mappings ${{\Phi }_{i}}$ and adding, we obtain
	\[\begin{aligned}
	f'\left( s \right)\sum\limits_{i=1}^{n}{{{\Phi }_{i}}\left( {{A}_{i}} \right)}-f'\left( s \right)s{{\mathbf{1}}_{\mathcal{K}}}&\le \sum\limits_{i=1}^{n}{{{\Phi }_{i}}\left( f\left( {{A}_{i}} \right) \right)}-f\left( s \right){{\mathbf{1}}_{\mathcal{K}}} \\ 
	& \le \sum\limits_{i=1}^{n}{{{\Phi }_{i}}\left( f'\left( {{A}_{i}} \right){{A}_{i}} \right)}-s\sum\limits_{i=1}^{n}{{{\Phi }_{i}}\left( f'\left( {{A}_{i}} \right) \right),}  
	\end{aligned}\] 
which implies
	\[\begin{aligned}
	f'\left( s \right)\left\langle \sum\limits_{i=1}^{n}{{{\Phi }_{i}}\left( {{A}_{i}} \right)}x,x \right\rangle -f'\left( s \right)s&\le \left\langle \sum\limits_{i=1}^{n}{{{\Phi }_{i}}\left( f\left( {{A}_{i}} \right) \right)}x,x \right\rangle -f\left( s \right) \\ 
	& \le \left\langle \sum\limits_{i=1}^{n}{{{\Phi }_{i}}\left( f'\left( {{A}_{i}} \right){{A}_{i}} \right)}x,x \right\rangle -s\left\langle \sum\limits_{i=1}^{n}{{{\Phi }_{i}}\left( f'\left( {{A}_{i}} \right) \right)}x,x \right\rangle   
	\end{aligned}\] 
	for any unit vector $x\in \mathcal{K}$.\\
	Applying again functional calculus for the operator $\sum\nolimits_{i=1}^{n}{{{\Phi }_{i}}\left( {{B}_{i}} \right)}$, we infer  
	\[\begin{aligned}
	& \left\langle \sum\limits_{i=1}^{n}{{{\Phi }_{i}}\left( {{A}_{i}} \right)}x,x \right\rangle f'\left( \sum\limits_{i=1}^{n}{{{\Phi }_{i}}\left( {{B}_{i}} \right)} \right)-f'\left( \sum\limits_{i=1}^{n}{{{\Phi }_{i}}\left( {{B}_{i}} \right)} \right)\sum\limits_{i=1}^{n}{{{\Phi }_{i}}\left( {{B}_{i}} \right)} \\ 
	& \le \left\langle \sum\limits_{i=1}^{n}{{{\Phi }_{i}}\left( f\left( {{A}_{i}} \right) \right)}x,x \right\rangle {{\mathbf{1}}_{\mathcal{K}}}-f\left( \sum\limits_{i=1}^{n}{{{\Phi }_{i}}\left( {{B}_{i}} \right)} \right) \\ 
	& \le \left\langle \sum\limits_{i=1}^{n}{{{\Phi }_{i}}\left( f'\left( {{A}_{i}} \right){{A}_{i}} \right)}x,x \right\rangle {{\mathbf{1}}_{\mathcal{K}}}-\left\langle \sum\limits_{i=1}^{n}{{{\Phi }_{i}}\left( f'\left( {{A}_{i}} \right) \right)}x,x \right\rangle \sum\limits_{i=1}^{n}{{{\Phi }_{i}}\left( {{B}_{i}} \right)}.  
	\end{aligned}\] 
	Thus, for any unit vector $x\in \mathcal{K}$,
	\begin{align}
	& \left\langle \sum\limits_{i=1}^{n}{{{\Phi }_{i}}\left( {{A}_{i}} \right)}x,x \right\rangle \left\langle f'\left( \sum\limits_{i=1}^{n}{{{\Phi }_{i}}\left( {{B}_{i}} \right)} \right)x,x \right\rangle -\left\langle f'\left( \sum\limits_{i=1}^{n}{{{\Phi }_{i}}\left( {{B}_{i}} \right)} \right)\sum\limits_{i=1}^{n}{{{\Phi }_{i}}\left( {{B}_{i}} \right)}x,x \right\rangle  \nonumber\\ 
	& \le \left\langle \sum\limits_{i=1}^{n}{{{\Phi }_{i}}\left( f\left( {{A}_{i}} \right) \right)}x,x \right\rangle -\left\langle f\left( \sum\limits_{i=1}^{n}{{{\Phi }_{i}}\left( {{B}_{i}} \right)} \right)x,x \right\rangle  \nonumber\\ 
	& \le \left\langle \sum\limits_{i=1}^{n}{{{\Phi }_{i}}\left( f'\left( {{A}_{i}} \right){{A}_{i}} \right)}x,x \right\rangle -\left\langle \sum\limits_{i=1}^{n}{{{\Phi }_{i}}\left( f'\left( {{A}_{i}} \right) \right)}x,x \right\rangle \left\langle \sum\limits_{i=1}^{n}{{{\Phi }_{i}}\left( {{B}_{i}} \right)}x,x \right\rangle   \nonumber
	\end{align}
	as required.
\end{proof}
\subsection{Function order without prior order assumptions}

In this part of the paper, we employ Theorem \ref{05} to  obtain several order results for any set of certain operators. We begin with the following generalized version of  \eqref{choi_ineq_intro} and \eqref{7} for different $n-$tuples of operators.

\begin{corollary}\label{8}
	Let all the assumptions of Theorem \ref{05} be satisfied. Then 
	\begin{equation}\label{07}
	f\left( \sum\limits_{i=1}^{n}{{{\Phi }_{i}}\left( {{B}_{i}} \right)} \right)\le \sum\limits_{i=1}^{n}{{{\Phi }_{i}}\left( f\left( {{A}_{i}} \right) \right)}+\delta\mathbf{1}_\mathcal{K},
	\end{equation}
	where
	\begin{small}
		\[\delta =\underset{\left\| x \right\|=1}{\mathop{\underset{x\in \mathcal{K}}{\mathop{\sup }}\,}}\,\left\{ \left\langle f'\left( \sum\limits_{i=1}^{n}{{{\Phi }_{i}}\left( {{B}_{i}} \right)} \right)\sum\limits_{i=1}^{n}{{{\Phi }_{i}}\left( {{B}_{i}} \right)}x,x \right\rangle -\left\langle \sum\limits_{i=1}^{n}{{{\Phi }_{i}}\left( {{A}_{i}} \right)}x,x \right\rangle \left\langle f'\left( \sum\limits_{i=1}^{n}{{{\Phi }_{i}}\left( {{B}_{i}} \right)} \right)x,x \right\rangle  \right\}.\]
	\end{small}
\end{corollary}
\begin{proof}
	From \eqref{02}, it follows that for any unit vector $x\in\mathcal{K}$,
	\[\begin{aligned}
	& \left\langle f\left( \sum\limits_{i=1}^{n}{{{\Phi }_{i}}\left( {{B}_{i}} \right)} \right)x,x \right\rangle  \\ 
	& \le \left\langle \sum\limits_{i=1}^{n}{{{\Phi }_{i}}\left( f\left( {{A}_{i}} \right) \right)}x,x \right\rangle  \\ 
	&\quad +\left\langle f'\left( \sum\limits_{i=1}^{n}{{{\Phi }_{i}}\left( {{B}_{i}} \right)} \right)\sum\limits_{i=1}^{n}{{{\Phi }_{i}}\left( {{B}_{i}} \right)}x,x \right\rangle -\left\langle \sum\limits_{i=1}^{n}{{{\Phi }_{i}}\left( {{A}_{i}} \right)}x,x \right\rangle \left\langle f'\left( \sum\limits_{i=1}^{n}{{{\Phi }_{i}}\left( {{B}_{i}} \right)} \right)x,x \right\rangle  \\ 
	& \le \left\langle \sum\limits_{i=1}^{n}{{{\Phi }_{i}}\left( f\left( {{A}_{i}} \right) \right)}x,x \right\rangle +\delta .  
	\end{aligned}\]
	This implies the desired inequality.
\end{proof}

If we let $A_i=B_i$ in Corollary \ref{8}, we obtain the following convex version of \eqref{7}.
\begin{corollary}\label{6}
	Let ${{\Phi }_{1}},\ldots ,{{\Phi }_{n}}:\mathcal{B}\left( \mathcal{H} \right)\to \mathcal{B}\left( \mathcal{K} \right)$ be positive linear mappings with $\sum\nolimits_{i=1}^{n}{{{\Phi }_{i}}\left( {{\mathbf{1}}_{\mathcal{H}}} \right)}={{\mathbf{1}}_{\mathcal{K}}}$ and let ${{A}_{1}},\ldots ,{{A}_{n}}\in \mathcal{B}\left( \mathcal{H} \right)$ be self adjoint operators with  spectra contained in the interval $J$. If $f\in {{C}^{1}_{\text{co}}}\left( J \right)$, then
	\[f\left( \sum\limits_{i=1}^{n}{{{\Phi }_{i}}\left( {{A}_{i}} \right)} \right)\le \sum\limits_{i=1}^{n}{{{\Phi }_{i}}\left( f\left( {{A}_{i}} \right) \right)}+\eta {{\mathbf{1}}_{\mathcal{K}}},\]
	where
	\begin{small}
		\[\eta =\underset{\left\| x \right\|=1}{\mathop{\underset{x\in \mathcal{K}}{\mathop{\sup }}\,}}\,\left\{ \left\langle f'\left( \sum\limits_{i=1}^{n}{{{\Phi }_{i}}\left( {{A}_{i}} \right)} \right)\sum\limits_{i=1}^{n}{{{\Phi }_{i}}\left( {{A}_{i}} \right)}x,x \right\rangle -\left\langle \sum\limits_{i=1}^{n}{{{\Phi }_{i}}\left( {{A}_{i}} \right)}x,x \right\rangle \left\langle f'\left( \sum\limits_{i=1}^{n}{{{\Phi }_{i}}\left( {{A}_{i}} \right)} \right)x,x \right\rangle  \right\}.\]	
	\end{small}
\end{corollary}

Further simpler, if we let $n=1$ in Corollary \ref{8}, we reach the following order relation between $f(A)$ and $f(B)$ in general.
\begin{corollary}\label{cor_order_A,B}
	Let  $A,B\in \mathcal{B}\left( \mathcal{H} \right)$ be self adjoint operators with  spectra contained in the interval $J$ and let $f\in {{C}^{1}_{\text{co}}}\left( J \right)$. Then
	$$f\left( B \right)\le f\left( A \right)+\gamma {{\mathbf{1}}_{\mathcal{H}}},$$
	where
	$$\gamma =\underset{\left\| x \right\|=1}{\mathop{\underset{x\in \mathcal{H}}{\mathop{\sup }}\,}}\,\left\{ \left\langle f'\left( B \right)Bx,x \right\rangle -\left\langle Ax,x \right\rangle \left\langle f'\left( B \right)x,x \right\rangle  \right\}.$$
\end{corollary}

Next, we show a reversed version of \eqref{07}.
\begin{corollary}\label{800}
Let all assumptions of Theorem \ref{8} hold. Then
\begin{equation}\label{400}
\sum\limits_{i=1}^{n}{{{\Phi }_{i}}\left( f\left( {{A}_{i}} \right) \right)}\le f\left( \sum\limits_{i=1}^{n}{{{\Phi }_{i}}\left( {{B}_{i}} \right)} \right)+\theta {{\mathbf{1}}_{\mathcal{K}}},
\end{equation}
where
\[\theta =\underset{\left\| x \right\|=1}{\mathop{\underset{x\in \mathcal{K}}{\mathop{\sup }}\,}}\,\left\{ \left\langle \sum\limits_{i=1}^{n}{{{\Phi }_{i}}\left( f'\left( {{A}_{i}} \right){{A}_{i}} \right)}x,x \right\rangle -\left\langle \sum\limits_{i=1}^{n}{{{\Phi }_{i}}\left( f'\left( {{A}_{i}} \right) \right)}x,x \right\rangle \left\langle \sum\limits_{i=1}^{n}{{{\Phi }_{i}}\left( {{B}_{i}} \right)}x,x \right\rangle  \right\}.\]
\end{corollary}

\begin{proof}

It follows from the relation \eqref{3},
\[\begin{aligned}
 \left\langle \sum\limits_{i=1}^{n}{{{\Phi }_{i}}\left( f\left( {{A}_{i}} \right) \right)}x,x \right\rangle &\le \left\langle f\left( \sum\limits_{i=1}^{n}{{{\Phi }_{i}}\left( {{B}_{i}} \right)} \right)x,x \right\rangle  \\ 
&\quad +\left\langle \sum\limits_{i=1}^{n}{{{\Phi }_{i}}\left( f'\left( {{A}_{i}} \right){{A}_{i}} \right)}x,x \right\rangle -\left\langle \sum\limits_{i=1}^{n}{{{\Phi }_{i}}\left( f'\left( {{A}_{i}} \right) \right)}x,x \right\rangle \left\langle \sum\limits_{i=1}^{n}{{{\Phi }_{i}}\left( {{B}_{i}} \right)}x,x \right\rangle . \\ 
\end{aligned}\]
for any unit vector $x\in\mathcal{K}$. Thus,
\begin{equation}\label{4}
\sum\limits_{i=1}^{n}{{{\Phi }_{i}}\left( f\left( {{A}_{i}} \right) \right)}\le f\left( \sum\limits_{i=1}^{n}{{{\Phi }_{i}}\left( {{B}_{i}} \right)} \right)+\theta {{\mathbf{1}}_{\mathcal{K}}},
\end{equation}
where
\[\theta =\underset{\left\| x \right\|=1}{\mathop{\underset{x\in \mathcal{K}}{\mathop{\sup }}\,}}\,\left\{ \left\langle \sum\limits_{i=1}^{n}{{{\Phi }_{i}}\left( f'\left( {{A}_{i}} \right){{A}_{i}} \right)}x,x \right\rangle -\left\langle \sum\limits_{i=1}^{n}{{{\Phi }_{i}}\left( f'\left( {{A}_{i}} \right) \right)}x,x \right\rangle \left\langle \sum\limits_{i=1}^{n}{{{\Phi }_{i}}\left( {{B}_{i}} \right)}x,x \right\rangle  \right\}\]
yields the assertion.
\end{proof}

Of course, \eqref{4} implies
\[\sum\limits_{i=1}^{n}{{{\Phi }_{i}}\left( f\left( {{A}_{i}} \right) \right)}\le f\left( \sum\limits_{i=1}^{n}{{{\Phi }_{i}}\left( {{A}_{i}} \right)} \right)+\vartheta {{\mathbf{1}}_{\mathcal{K}}},\]
where
\[\vartheta =\underset{\left\| x \right\|=1}{\mathop{\underset{x\in \mathcal{K}}{\mathop{\sup }}\,}}\,\left\{ \left\langle \sum\limits_{i=1}^{n}{{{\Phi }_{i}}\left( f'\left( {{A}_{i}} \right){{A}_{i}} \right)}x,x \right\rangle -\left\langle \sum\limits_{i=1}^{n}{{{\Phi }_{i}}\left( f'\left( {{A}_{i}} \right) \right)}x,x \right\rangle \left\langle \sum\limits_{i=1}^{n}{{{\Phi }_{i}}\left( {{A}_{i}} \right)}x,x \right\rangle  \right\},\]
which is a reverse of \eqref{7}.

\subsection{Function order with prior order assumptions}
To easily understand the discussion of this part of the paper, we begin with the following remark.

\begin{remark}
It is a valid question to ask about the quantities $\delta$ in Corollary \ref{8}, $\eta$ in Corollary \ref{6} and $\gamma$ in Corollary \ref{cor_order_A,B}. In particular, are these quantities positive or negative?\\
The purpose of this remark is to discuss and answer this question. Applying functional calculus for $s=A$ in \eqref{grad_ineq_intro}, we obtain
$$f(A)-f(t)\mathbf{1}_\mathcal{H}\leq Af'(A)-tf'(A),$$ which implies
$$\left<f(A)x,x\right>-f(t)\leq \left<Af'(A)x,x\right>-t\left<f'(A)x,x\right>,~ x\in\mathcal{H}, \|x\|=1.$$ Now replacing $t$ by $\left<Ax,x\right>$ and noting \eqref{19} (for $n=1$ and $\Phi$ being the identity mapping), we obtain
$$ \left<Af'(A)x,x\right>-\left<Ax,x\right>\left<f'(A)x,x\right>\geq \left<f(A)x,x\right>-f\left(\left<Ax,x\right>\right)\geq 0.$$
Therefore, we have shown that if $f\in C^{1}_{\text{co}}(J)$ and $A$ is a self adjoint operator with spectrum in $J$, then
$$ \left<Af'(A)x,x\right>-\left<Ax,x\right>\left<f'(A)x,x\right>\geq 0$$ for any unit vector $x\in\mathcal{H}.$\\
Consequently, replacing $A$ by $\sum_{i=1}^{n}\Phi_i(A_i)$, we infer that $\eta\geq 0$ in Corollary \ref{6}.

Furthermore, if $A\leq B$ and $f'\geq 0$ (i.e., $f$ is increasing), then $\left<Ax,x\right>\leq \left<Bx,x\right>$ and since $f'\geq 0,$ we deduce
\begin{align}\label{needed_1}
\left\langle f'\left( B \right)Bx,x \right\rangle -\left\langle Ax,x \right\rangle \left\langle f'\left( B \right)x,x \right\rangle&\geq \left\langle f'\left( B \right)Bx,x \right\rangle -\left\langle Bx,x \right\rangle \left\langle f'\left( B \right)x,x \right\rangle\geq 0,
\end{align}
which shows that $\gamma\geq 0$ in Corollary \ref{cor_order_A,B} when $A\leq B$ and $f$ is increasing.\\
On the other hand, if $A\geq B$ and $f'\leq 0$ (i.e., $f$ is decreasing), then
\begin{align*}
\left\langle f'\left( B \right)Bx,x \right\rangle -\left\langle Ax,x \right\rangle \left\langle f'\left( B \right)x,x \right\rangle&\geq \left\langle f'\left( B \right)Bx,x \right\rangle -\left\langle Bx,x \right\rangle \left\langle f'\left( B \right)x,x \right\rangle\\
&\geq 0,
\end{align*}
which again shows that $\gamma\geq 0$ in Corollary \ref{cor_order_A,B} when $A\geq B$ and $f$ is decreasing.\\
Finally, in \eqref{needed_1}, if we let $B=\sum_{i=1}^{n}\Phi_i(B_i)$ and $A=\sum_{i=1}^{n}\Phi_i(A_i),$ and assume that $\sum_{i=1}^{n}\Phi_i(A_i)\leq \sum_{i=1}^{n}\Phi_i(B_i)$, we infer that $\delta\geq 0$ in Corollary \ref{8}, when $f$ is increasing.

Of course, the question is still valid to ask if it is possible to have negative values for those quantities. The answer is yes! For example, if $f(t)=t$ and if we have $A\geq B$, then
\begin{align*}
\left\langle f'\left( B \right)Bx,x \right\rangle -\left\langle Ax,x \right\rangle \left\langle f'\left( B \right)x,x \right\rangle&=\left<Bx,x\right>-\left<Ax,x\right>\leq 0;
\end{align*}
showing that $\gamma\leq 0$ in Corollary \ref{cor_order_A,B} for some cases. The other quantities can be treated similarly.
\end{remark}

We conclude this article with the more elaborated versions of Corollary \ref{cor_order_A,B}, which read as follows. In both results $\gamma$ is still as in Corollary \ref{cor_order_A,B}.
\begin{corollary}\label{cor_order_A,B_2}
	Let  $A,B\in \mathcal{B}\left( \mathcal{H} \right)$ be self adjoint operators with  spectra contained in the interval $J$ and let $f\in {{C}^{1}_{\text{co}}}\left( J \right)$. If $A\leq B$ and $f$ is increasing, then a non-negative number $\gamma$ exists such that
	$$f\left( B \right)\le f\left( A \right)+\gamma {{\mathbf{1}}_{\mathcal{H}}}.$$
\end{corollary}

\begin{corollary}\label{cor_order_A,B_3}
	Let  $A,B\in \mathcal{B}\left( \mathcal{H} \right)$ be self adjoint operators with  spectra contained in the interval $J$ and let $f\in {{C}^{1}_{\text{co}}}\left( J \right)$. If $B\leq A$ and $f$ is decreasing, then a non-negative number $\gamma$ exists such that
	$$f\left( B \right)\le f\left( A \right)+\gamma {{\mathbf{1}}_{\mathcal{H}}}.$$
	
\end{corollary}
Therefore, Corollaries \ref{cor_order_A,B_2} and \ref{cor_order_A,B_3} present order reversing  and order preserving  results, respectively.\\
It should be noticed that, for example, when $B\leq A$, we still have the order $f\left( B \right)\le f\left( A \right)+\gamma {{\mathbf{1}}_{\mathcal{H}}}$ for the increasing function $f$, which presents an order preserving result for the increasing $f$. However, in this case, we do not have any information about whether $\gamma\geq 0$ or not.

Such order preserving or reversing results then can be found similarly for the quantities appearing in Corollary \ref{8} and \ref{800}.

We emphasize that a major difference between the above results and those in Theorems \ref{th1.2} and \ref{c} is the fact that our results do not assume $\sigma(A),\sigma(B)\subset [m,M]$ for some scalars $m,M.$ Although the literature is rich in studying such inequalities, we are not aware of any results that treat arbitrary self adjoint operators without the restriction $\sigma(A),\sigma(B)\subset [m,M]$.

\section*{Acknowledgments}
The authors thank an anonymous referee for his/her insightful comments and suggestions.

\vskip 0.5 true cm

\vskip 0.3 true cm

{\tiny (G. Karamali) Faculty of Basic Sciences, Shahid Sattari Aeronautical University of Science and Technology, South Mehrabad, Tehran, Iran.
	
	\textit{E-mail address:} g\_karamali@iust.ac.ir}

\vskip 0.3 true cm

{\tiny (H.R. Moradi) Faculty of Basic Sciences, Shahid Sattari Aeronautical University of Science and Technology, South Mehrabad, Tehran, Iran.}

{\tiny \textit{E-mail address:} hrmoradi@mshdiau.ac.ir}

\vskip 0.3 true cm 	 

{\tiny (M. Sababheh) Department of Basic Sciences, Princess Sumaya University for Technology, Amman 11941,
	Jordan. 
	
	\textit{E-mail address:} sababheh@yahoo.com; sababheh@psut.edu.jo}

\begin{thebibliography}{9}
\bibitem{bhatia}
R. Bhatia, {\it Matrix analysis}, Springer Verlag, New York, 1997.

\bibitem{choi}
 M. D. Choi, {\it A Schwarz inequality for positive linear maps on $C^*-$algebras}, Illinois J. Math., {\bf18} (1974), 565--574.
 
\bibitem{davis}
C. Davis, {\it A Schwarz inequality for convex operator functions}, Proc. Amer. Math. Soc., {\bf 8} (1957), 42–-44.

\bibitem{n5}
S. Furuichi, H. R. Moradi and M. Sababheh, {\it New sharp inequalities for operator means}, Linear Multilinear Algebra., {\bf67}(8) (2019), 1567--1578.

\bibitem{5}
S. Furuichi, H. R. Moradi and A. Zardadi, {\it Some new Karamata type inequalities and their applications to some entropies}, Rep. Math. Phys., {\bf84}(2)  (2019), 201--214.

\bibitem{n3}
T. Furuta, {\it Operator inequalities associated with H\"older--McCarthy and Kantorovich inequalities}, J. Inequal. Appl., {\bf2} (1998), 137--148.

\bibitem{book1}
T. Furuta, J. Mi\'ci\'c, J. Pe\v cari\'c and Y. Seo, {\it Mond--Pe\v cari\'c method in operator inequalities}, Element, Zagreb, 2005.


\bibitem{n4}
I. H. G\"um\"u\c s, H. R. Moradi and M. Sababheh, {\it More accurate operator means inequalities}, J. Math. Anal. Appl., {\bf465}(1) (2018), 267--280.

\bibitem{1}
F. Hansen, J. Pe\v cari\'c and I. Peri\'c, {\it Jensen's operator inequality and it's converses}, Math. Scand., {\bf100} (2007), 61--73.


\bibitem{10}
L. Horv\'ath, K. A. Khan and J. Pe\v cari\'c, {\it Cyclic refinements of the different versions of operator Jensen's inequality}, Electron. J. Linear Algebra., {\bf31}(1) (2016), 125--133.

 
\bibitem{6}
J. Mi\'ci\'c, H. R. Moradi and S. Furuichi, {\it Choi--Davis--Jensen's inequality without convexity}, J. Math. Inequal., {\bf12}(4) (2018), 1075--1085.

\bibitem{re2}
J. Mi\'ci\'c and H. R. Moradi, {\it Some inequalities involving operator means and monotone convex functions}, J. Math. Inequal., {\bf14} (2020), 135--145. 

\bibitem{11}
 J. Mi\'ci\'c and J. Pe\v cari\'c, {\it Some mappings related to Levinson's inequality for Hilbert space operators}, Filomat., {\bf31}  (2017), 1995--2009.
 



\bibitem{n1}
J. Mi\'ci\'c, J. Pe\v cari\'c and Y. Seo, {\it Function order of positive operators based on the Mond--Pe\v cari\'c method}, Linear Algebra Appl., {\bf360} (2003), 15--34.

\bibitem{12}
B. Mond and J. Pe\v cari\'c, {\it On Jensen's inequality for operator convex functions}, Houston J. Math., {\bf21} (1995), 739--753.

\bibitem{n2}
J. Pe\v cari\'c and J. Mi\'ci\'c, {\it Some functions reversing the order of positive operators}, Linear Algebra Appl., {\bf396} (2005),  175--187.



\bibitem{re1}
M. Sababheh, H. R. Moradi and S. Furuichi, {\it Reversing Bellman operator inequality}, J. Math. Inequal., {\bf14} (2020), 577--584.

\bibitem{re3}
M. Shah Hosseini, H. R. Moradi and B. Moosavi, 
 {\it Operator Jensen's type inequalities for convex functions}, J. Math. Ext., to appear.

\end{thebibliography}
\end{document}